\documentclass[a4paper,12pt,reqno]{amsart}   
\usepackage{amsmath,amsthm,amssymb,amscd}
\usepackage{mathrsfs,cancel}
\usepackage[dvips]{graphics}
\usepackage{enumitem,color}

\newcommand{\K}{\mathbb{K}}

\newcommand{\diam}{\mathop{\mbox{\rm diam}}\nolimits}

\def\vol#1{\textbf{#1}}

\def\textindent[#1]{\indent\llap{\mbox{\rm #1}\enspace}\ignorespaces}
\def\btextindent{\indent\llap{\mbox{$\bullet$}\enspace}\ignorespaces}

\makeatletter
\def\@cite#1#2{{%
 \m@th\upshape\mdseries[{\mdseries #1}{\if@tempswa, #2\fi}]}}
\@ifundefined{cite }{%
  \expandafter\let\csname cite \endcsname\cite
  \edef\cite{\@nx\protect\@xp\@nx\csname cite \endcsname}%
}{}

\def\@seccntformat#1{%
  \protect\textup{\bfseries\protect\@secnumfont
    \csname the#1\endcsname
    \protect\@secnumpunct
  }%
}

\def\section{\@startsection{section}{1}%
  \z@{.7\linespacing\@plus\linespacing}{.5\linespacing}%
  {\normalfont\bfseries
}}

\def\subsection{\@startsection{subsection}{2}%
  \z@{.5\linespacing\@plus.7\linespacing}{.5\linespacing}%
  {\normalfont\it
}} 

\renewenvironment{proof}[1][\proofname]{\par
  \pushQED{\qed}%
  \normalfont \topsep6\p@\@plus6\p@\relax
  \trivlist
  \itemindent\normalparindent
  \item[\hspace*{-\parindent}
        \scshape
    #1\@addpunct{.}]\ignorespaces
}{%
  \popQED\endtrivlist\@endpefalse
}

\makeatother

\newtheoremstyle{definice}
  {3pt}
  {3pt}
  {\upshape}
  {}
  {\scshape}
  {.}
  {.5em}
  {}

\theoremstyle{definice}

\newtheorem{defi}{Definition}
\newtheorem{rema}{Remark}
\newtheorem{exa}{Example}

\newtheoremstyle{vety}
  {3pt}
  {3pt}
  {\itshape}
  {}
  {\scshape}
  {.}
  {.5em}
  {}

\theoremstyle{vety}

\newtheorem{lema}{Lemma}
\newtheorem{thm}{Theorem}
\newtheorem{coro}{Corollary}

\def\R{\mathbb{R}}
\def\N{\mathbb{N}}

\def\B{\mathcal{B}}
\def\O{\mathcal{O}}
\def\Ne{\mathcal{N}}
\def\U{\mathcal{U}}

\newcounter{tabulky} 

\makeindex          

\begin{document}    

\title[Stability of multivalued attractors]
{Stability of multivalued attractors}

\author{Miroslav Rypka}\thanks{Supported by the project StatGIS Team No.\ CZ 1.07/2.3.00/20.0170.}
\maketitle

\noindent
{\small Dept. of Math. Anal. and Appl. of Math., Faculty of Science,\newline 
Palack\'{y} University, 17.~listopadu 12, 771~46 Olomouc, Czech Republic \newline
\centerline{e-mail:  miroslav.rypka01@upol.cz}}
\newline
\newline
 {\bf Abstract.} \newline
Stimulated by recent problems in the theory of iterated function systems, we provide a variant of the Banach converse theorem for multivalued maps. In particular, we show that attractors of continuous multivalued maps in a metric space are stable. Moreover, such attractors in locally compact, complete metric spaces may be obtained by means of the Banach theorem in the hyperspace.
\newline
Keywords and phrases: \newline
\newline
AMS Subject Classification:

\section{Introduction}
Multivalued maps and their attractors are studied in relation to dynamical systems, e.g. iterated function systems or differential inclusions. Throughout the whole paper, we consider continuous multivalued maps with compact values which generate  continuous operators on hyperspaces, as discussed in the next section. 

Our motivation is following.
We would like to state a variant of J\'{a}no\v{s} theorem for operators on hyperspaces induced by  multivalued maps. Under J\'{a}no\v{s} theorem we understand the results on the converse of Banach theorem developed in \cite{Ja}, \cite{Me}, \cite{Jach}, \cite{Le1}, \cite{Le2}, \cite{Op}. In spite of the metric nature of the Banach theorem, these papers provide several topological conditions on a map to be contractive. 

Although, the theory  of the converse to the Banach theorem seems complete, analogical problems in the theory of multivalued maps (for detailed treatment of attractors of multivalued maps see e.g. \cite{AF}, \cite{AFGL}, \cite{Na}) and iterated function systems are still addressed. Since the Hutchinson's seminal work \cite{Hu} (see also \cite{Wi}), the metric approach to attractors of IFSs dominated. With only a few exceptions (\cite{Ki}, \cite{LM1}, \cite{LM2}, \cite{LM3}), the attractors of IFS were obtained by means of the Banach theorem. 
Recently, it was pointed out \cite{ABVW}, \cite{BLR1}, \cite{BLR2} that the attractor of IFS 
is a topological notion and the contractivity of maps in an IFS is only a sufficient condition for the existence of an attractor (for different approaches see e.g. \cite{BN}, \cite{Mi}). Novelty of this fact is caused by the prevailing interest in affine IFSs in Euclidean spaces, for which the existence of an attractor is equivalent to the existence of equivalent metric in original space w.r.t. which the maps in IFS are contractions \cite{ABVW}. 


The question whether it holds for any IFS with point fibred attractor $A$ was raised by 
Kameyama (\cite{Ka}). \emph{Does there exist a metric on  $A$ such that Hutchinson operartor $F|_A$ is contractive and the topology on $A$ induced by this metric is the same as the topology of $X$ restricted to $A?$}

Similar problem for multivalued maps was stated by 
Fryszkowski (\cite{Jach}). \emph{Let $X$ be an arbitrary nonempty set, $2^X$ be the family of all nonempty subsets of $X$ and $F:X\rightarrow 2^X$ be a multivalued map. Find necessary conditions and (or) sufficient conditions for the existence of a complete metric $d$ for $X$ such that given $c\in(0,1),$ $F$ would be a Nadler (\cite{Na}) multivalued $c-$contraction with respect to $d,$ that is}
$$d_H(F(x),F(y))\leq cd(x,y) \forall x,y\in X,$$
\emph{where $d_H$ denotes the Hausdorff metric generated by $d$.} 

In contrast to Fryszkowski problem or Kameyama question, we shift the search for the metric w.r.t. which a map is contracting to the hyperspace. In particular, we will explore the operator in hyperspace induced by multivalued map. We will proceed in the following way. Next section recalls the basic notions, e.g. multivalued maps, attractors and strict attractors. Main results can be found in Section 3.
In Theorem \ref{th1} we prove that attractors of multivalued maps are stable fixed points of associated operators in hyperspace.  The stability of the attactor is implied by the monotonicity of such operators.
Corollary \ref{co2}  provides a variant of  J\'{a}no\v{s} theorem for operators generated by multivalued maps in locally compact, complete metric  spaces.
The same conditions as in Theorem \ref{th1} imply also the stability of strict attractors. Hence, we express the analogical results for strict attractors in Corollary \ref{co1} and \ref{co3}. Finally, a few examples are provided. Exaple 2 shows that we cannot drop monotonicity condition. Operators in a hypperspace need not be generated by multivalued maps. Attractivity of multivalued operators does not imply their  contractivity even in compact spaces. Example 3 illustrates relation of our theory to Fryszkowski problem. It proves that multivalued maps generating contractive operators need not be contractions, even in compact metric spaces. 

\section{Notation}
Throughout the whole paper, we deal  with a metric space $(X,d)$. Let us denote by $\K(X)$ the space of compact subsets of $X,$ called the \emph{hyperspace,} endowed usually with the Hausdorff metric $d_H$ defined  (cf. e.g. \cite{Hu}) 
$$
d_H(A,\,B):=\inf\{r>0;\,A\subset O_r(B) \mbox{ and }  B\subset O_r(A)\},
$$ 
where  $O_r(A):=\{x\in X;\, \exists a\in A:d(x,\,a)<r)\}$
and $A,\,B\in \K(X)$.
An alternative definition reads
$$d_H(A,\,B):=\max\{\sup_{a\in A} d(a,B),\sup_{b\in B} d(b,A)\} $$
$$=\max\{\sup_{a\in A} (\inf_{b\in B} d(a,b)),\,\sup_{b\in B} (\inf_{a\in A} d(a,b))\}.$$
\begin{rema}
The Hausdorff metric is induced by $d.$ However, there exist metrics in $\K(X)$ which cannot be induced by any $d$.
\end{rema}
Thus, we will denote a general metric in $\K(X)$ by $D$.

We will often employ a neighbourhood of a compact set as a point in a hyperspace. 
\begin{defi} Let $A\in (\K(X),d_H)$. We will write
$\Ne_r(A):=\{B\in \K(X);\, d_H(A,\,B)<r\}$
\end{defi}
In general, letters such as $\Ne,\O,\B\dots $ will stand for classes of sets.

\begin{defi}
The map $F:X\rightarrow \K(X)$ is called \emph{multivalued map} and the operator $F:\K(X)\rightarrow \K(X),$ defined by
$$F(A)=\bigcup_{x\in A}F(x),$$
is called \emph{multivalued operator.}
\end{defi}
In the paper, all the multivalued maps and operators are continuous w.r.t. $d$ and $d_H$.
Observe that continuous multivalued map on a metric space generates continuous multivalued operator (cf. \cite{AF}).

\begin{defi}
Let $(X,d)$ be a metric space. A map $f:X\rightarrow X$ is a contraction if for some 
$c\in[0,1),$ $d(f(x),f(y))\leq c d(x,y)$ for any $x,y\in X$.
\end{defi}

Multivalued maps and operators are often generated by iterated function systems (IFSs).
\begin{defi}
\emph{An Iterated function system} consists of finite number of continuous maps $f_i,\,i=1,2,\dots,N,$ on a metric space $(X,d)$.
\end{defi}

Any IFS yields a multivalued map $F:X\rightarrow \K(X),$
$$F(x)=\bigcup_{i=1}^{N}f_i(x),
$$
and a multivalued operator $F:\K(X)\rightarrow \K(X),$
$$F(A)=\bigcup_{x\in A}F(x),
$$
called \emph{Hutchinson operator.}

Usually, IFSs of contractive maps $f_i$ on complete metric spaces are treated. They possess attractor $A^*\in \K(X)$ due to the Banach theorem. Notice that this theorem gives not only existence, but also attractivity and numerical stability necessary for visualization of the attractor.

Let us discuss the notion of attractor.
\begin{defi}
Let $(X,d)$ be a metric space and $f:X\rightarrow X$ continuous with a fixed point $x^*.$ The point $x^*$ is called \textit{attractive} if for all $x\in X,$
$$\lim_{n\to\infty}d(f^n(x),x^*)=0.$$
\end{defi}

\begin{defi}
Let $(X,d)$ be a metric space and $f:X\rightarrow X$ continuous with a fixed point $x^*\in X$. The point $x^*$ is called \textit{stable} if for any $\epsilon>0$ there exists $\delta>0$ such that
$$
d(f^n(x),x^*)<\epsilon \forall n\in \N,\, x\in X,\,d(x,x^*)<\delta.
$$
The fixed point $x^*$ is \textit{asymptotically stable} if it is attractive and stable.
\end{defi}

\begin{defi}
Let $(X,d)$ be a metric space and $F:X\rightarrow K(X)$ be a continuous multivalued map. Let $A^*\in\K(X)$ be such that $F(A^*)=A^*$ and $U\subset K(X)$ open be such that  
$\lim_{n\to\infty}d_H(F^n(B),A^*)=0\forall B\in U$. Then $A^*$ is called \textit{an attractor of multivalued map} $F$.
\end{defi}

One possible definition of  an attractor of IFS employs the previous definition.
\begin{defi}
Let $\{X;f_{1},f_{2},\dots,f_{N}\}$ be an IFS and $F$ its Hutchinson operator with a fixed point $A^*\in K(X).$ $A^*$ is attractive if there exists $\mathcal{U}\subset K(X)$ open such that forall $B\in \mathcal{U}$
$$F^n(B)\rightarrow A^*.$$ 
\end{defi}

However, this definition has its drawbacks in hyperspaces, as can be seen from the following example.
\begin{exa}
 The IFS
$\{([-1,1],d_{eucl});f(x)=\sqrt[3]{x}\}$ possesses three overlapping attractors
\begin{itemize}
\item[] $A_1=\{-1\}, \ U_1=K([-1,0))$
\item[] $A_2=\{1\}, \ U_2=K((0,1])$
\item[] $A_3=\{-1,1\}, \ U_3=K([-1,1]\backslash \{0\})\backslash(U_1\cup U_2)$
\end{itemize}

\end{exa}
In order the attractors do not overlap, we introduce the notion of strict attractor (see \cite{BLR1}, \cite{BLR2}). 
\begin{defi}
	A compact set $A^*\subset X$ is a \textit{strict attractor} of $F$, 
if there exists an open set $U\supset A^*$ such that 
$$
\forall{S\in K(U)}, F^{n}(S)\rightarrow A^*.
$$
The maximal open set $U$ with the above property is called
the \textit{basin of attraction} of the attractor $A^*$ 
(with respect to $F$)
and denoted by $B(A^*)$. 
\end{defi}
\begin{rema}
The existence of the maximal open set $U$ is proven in \cite{BLR2}.
\end{rema}

\begin{rema}
Strict attractor is topological invariant (\cite[Lemma 2.8]{BHR}) and it is an attractor.
\end{rema}

\section{Results}
\begin{thm}\label{th1}
Any attractor of continuous multivalued map on a metric space $(X,d)$ is asymptotically stable in $(\K(X),d_H)$.
\end{thm}

The proof of the theorem proceeds in two steps. First, we show stability in $\K(A^*)$. Then we extend it to the whole hyperspace $\K(X)$. 
\begin{lema}
Let $(A^*,d)$ be a compact metric space. Let $F:A^*\rightarrow \K(A^*)$ be such that $F(A^*)=A^*$ and $$\exists \epsilon>0:\,\lim_{n\to\infty}d_H(F^n(\{B\}),A^*)=0,\, \forall B\in \Ne_{\epsilon_{A^*}}(A^*).$$ Then $F$ is asymptotically stable.
\end{lema}
\begin{proof}
We only need to show the stability of $A^*$. If $A^*$ is a singleton, then it is obviously stable in $\K(A^*)$.

Suppose that $A^*$ is not a singleton, which means $\diam(A^*)>0$. We will proceed by contradiction. Assume $A^*$ is not stable. Then there exist $\epsilon>0$ and  a sequence 
$$
\{B_n\}\in \K(A^*),\,d_H(B_n,A^*)<\min\left\{\frac{1}{n},\epsilon_{A^*}\right\}
$$
such that
$$
\exists k_n\in \N:\,d_H(F^{k_n}(B_n),A^*)>\epsilon.
$$
Observe that $k_n\to \infty,$ otherwise the operator $F$ would not be continuous. 

Since $F$ is continuous, for any $n\in \N,$ there is an open set $\B_n\subset \K(A^*),$ such that 
$$\B_n=\{B\in \Ne_{\epsilon_{A^*}}(A^*);\,d_H(F^{k_n}(B),A^*)>\epsilon\}.$$

In the following part, we will employ the monotonicity of $F$. Notice that for any set $C\in \K(A^*),\,C\subset B\in \B_n $
\begin{equation}\label{eqmon}
d_H(F^{k_n}(C),A^*)>\epsilon
\end{equation}
since
$$d_H(F^{k_n}(B),A^*)\leq d_H(F^{k_n}(C),A^*).
$$

We will show that there exists a set $K\in \K(A^*)$ such that $K$ belongs to the subsequence $\overline{\B_n}$. 
Denote by $\widetilde{\B}$ the set $\{C\in \Ne_{\epsilon_{A^*}}(A^*);\exists B\in \B,B\subset C\}.$ The set $\widetilde{\B}$ is open for $\B\in \K(A^*)$ open. Furthermore, let $(\B'\cap \widetilde{\B})^{-1}$ stand for the set $\{B\in\B;\,\exists B'\in\B',\,B\subset B'\} $ which is again open for $\B,\B'$ open.

Observe that for any $n_1\in \N$ there exists $n_2\in \N$ such that $\B_{n_1n_2}:=(\B_{n_2}'\cap \widetilde{\B_{n_1}})^{-1}$ is nonempty and open. Since $\B_{n_1}$ is nonempty and open, there exists an open ball $\O_{n_1}\subset \B_{n_1}$ with radius $r_{n_1}$. For $n_2\in \N,\,n_2>\frac{1}{r_{n_1}},$ consider the set $\B_{n_2}$. The inequality  $d_H(B_{n_2},A^*)<r_{n_1}$ and (\ref{eqmon}) imply that $(\B_{n_2}\cap \widetilde{\B_{n_1}})^{-1}$ is nonempty. It is also open, since $\B_{n_1}$ and $\B_{n_2}$ are open. We will simplify the notation using $\B_{n_1n_2}$ instead of  $(\B_{n_2}\cap \widetilde{\B_{n_1}})^{-1}$. 

Again, since $\B_{n_1n_2}$ is open, there exists an open ball $\O_{n_1n2}\subset \B_{n_1n2}$ with radius $r_{n_1n_2}$ and $n_3\in \N$ such that
$$
\B_{n_1n_2n_3}:=(\B_{n_3}\cap \widetilde{\B_{n_1n_2}})^{-1}
$$ is nonempty and open.
Repeating this process to infinity, we obtain the sequence
\begin{equation}\label{eqseq}
\B_{n_1},\B_{n_1n_2},\B_{n_1n_2n_3},\B_{n_1,n_2,n_3n_4}, \dots
\end{equation}
Observe that the sequence is nested, i.e.
$$
\B_{n_1}\supset\B_{n_1n_2}\supset\B_{n_1n_2n_3}\supset\B_{n_1n_2n_3n_4}\supset \dots
$$
and for any $p\in \N,$
$$
d_H(F^{n^p}(B),A^*)>\epsilon,\,B\in \B_{n_1n_2\dots n_p}.
$$  

Last, consider the sequence of closures
\begin{equation}\label{eqseq1}
\overline{\B_{n_1}},\overline{\B_{n_1n_2}},\overline{\B_{n_1n_2n_3}},\overline{\B_{n_1n_2n_3n_4}}, \dots
\end{equation}
where $\overline{\B_{n_1\dots n_k}}$ is compact in $\K(A^*),$ for any $k$. Since the sequence (\ref{eqseq1}) is also nested,  there is a nonempty intersection 
$\overline{\B}=\bigcap_{i=1}^{\infty}\overline{\B_{n_1n_2\dots n_i}}$.
 Continuity of $F$ implies 
$$d_H(F^{n_i}(K),A^*)\geq\epsilon ,\,i\in\N,\,K\in\overline{\B},$$
which is a contradiction to the attractivity of $A^*$ in $\Ne_{\epsilon_{A^*}}.$  
\end{proof}

\begin{proof}(Continuation of proof of Theorem 1)
Let us proceed to the second part of the proof, where we will use uniform stability of $A^*$ in a compact subset of $\Ne_{\epsilon_A^*}$ and uniform continuity of $F$ in feasible compact subset of $\K(X)$. 
Assume that $A^*$ is not stable. Then there exists $\epsilon>0,$ a sequence of sets $B_n\in U,\,d_H(B_n,A^*)<\frac{1}{n}$ and a sequence $i_n\in \N$ such that
$$d_H(F^{i_n}(B_n),A^*)>\epsilon,\,\forall n\in \N.$$
For the sake of simplicity, let us denote $F^i(B_n)$  by $B^i_n$.

Observe that 
$$
\overline{\{B_n\}}=\bigcup_{n=1}^{\infty}B_n\cup A^*
$$ 
is a compact subset of $X$ as well as 
$$
\hat{B}:=\bigcup_{i=1}^{\infty}F^i(\overline{\{B_n\}}).$$

The compactness of $\hat{B}$ implies that $\forall \epsilon \exists m_0\in \N$
such that
$$d_H(F^i(\hat{B}),A^*)<\epsilon \forall i\geq i_0.$$
Similarly, from the monotonicity of $F,$ we have
$$\forall \epsilon \exists i_0\in \N\forall i\geq i_0\forall B^i_n\exists C^i_n\in K(A^*)\cap\U:
d_H(B_n^{i_n},C_n^i)<\epsilon.$$

Since $F$ is stable on $\Ne_{\epsilon_{A^*}},$ it is also uniformly stable  on any of its compact subsets. Consider closed neghbourhood $\U\in \K(A^*)$ such that $\U\subset \Ne_{\epsilon_{A^*}}$. 

Without loss of generality, let $\epsilon>0$ be such that $\Ne_{\epsilon}(A^*)\subset \U$. The uniform stability of $F$ implies
\begin{equation}\label{eqh1}
\forall \epsilon>0 \exists k_0\in \N:\,d_H(F^k(B),A^*)<\frac{\epsilon}{2},\,k\geq k_0,\,\forall B\in \U.
\end{equation}
The operator $F$ is uniformly continuous on any compact subset of $\K(X),$ which implies 
$\forall k\in \N \forall \epsilon>0 \exists \delta_0>0:$
\begin{equation}\label{eqh2}
d_H(F^k(B),F^k(B'))<\frac{\epsilon}{2},\,\forall B,\,B'\in \K(\hat{B}),\,d_H(B,B')<\delta_0.
\end{equation}

Let $k_0$ fulfill (\ref{eqh1}) and  $\delta_0$ fulfill (\ref{eqh2}) for $\delta_0$. 
From attractivity of $F$ and from $F^n(\hat{B})\to 0,$  we get 
\begin{equation}
\forall \delta>0 \exists n_0\in \N \forall i\in \N \forall B_n^i \exists C_n^i\in \K(A):d_{H}(B^i_n,C^i_n)<\delta. 
\end{equation}\label{eqmen0}
Notice that there always exists $C_n^i\in \K(A)$
\begin{equation}\label{eqmen}
d_H(C^i_n,A^*)\leq d_H(B^i_n,A^*).
\end{equation}

Consider a sequence $\{i_n\}\in \N$ such that
\begin{equation}\label{eqtr}
d_H(B_n^{i_n},A^*)\geq\epsilon
\end{equation}
and 
$$d_H(B_n^{i},A^*)\geq\epsilon\Rightarrow i>i_n.$$
Observe that $i_n\to \infty,$ otherwise $F$ would not be continuous.

Let us investigate the behaviour of $F^{k_0}$ on $B^{i_n-k_0}_n$ for $n\in \N$ such that $i_n>k_0$.
From (\ref{eqh1}) and (\ref{eqh2}), we obtain (notice that $d_H(C^{i_n-k_0}_n,A^*)<\epsilon$ is implied by  (\ref{eqmen0}) and  (\ref{eqmen})) 
$$
d_H(F^{k_0}(B^{i_n-k_0}_n),A^*)\leq d_H(F^{k_0}(B^{i_n-k_0}_n),C^{i_n-k_0}_n)+d_H(F^{k_0}(C^{i_n-k_0}_n),A^*)
$$
implying
$$
d_H(B^{i_n}_n,A^*)<\epsilon
$$
which is a contradiction to (\ref{eqtr}).
\end{proof}

Since a strict attractor is an attractor, we obtain:
\begin{coro}\label{co1}
Any strict attractor of continuous multivalued map on a metric space $(X,d)$ is asymptotically stable in $(\K(X),d_H)$.
\end{coro}

Adding condititions of completeness and local compactness on an original space $X,$ attractive multivalued  induce a contraction in hyperspace. 
Notice that a hyperspace $(\K(X),d_H)$ inherits most of the features of a metric space $(X,d)$. 
\begin{rema}(cf. e.g. \cite{AR})
Let $(X,d)$ be a metric space. The hyperspace $(\K(X),d_H)$ is complete if and only if $(X,d)$ is complete.  The hyperspace $(\K(X),d_H)$ is locally compact if and only if $(X,d)$ is locally compact.
\end{rema}

\begin{lema}(cf. e.g. \cite[Theorem 2.1]{Op})\label{reop}
Let $(X,d)$ be a   locally compact, complete metric space. Let $f:X\rightarrow X$ be a continuous map with a fixed point $x^*\in X$ such that $x^*$ is attractive and stable. Then there exists metric $d'$ equivalent to $d$ such that $f$ is a contraction in $(X,d')$.
\end{lema}

Theorem \ref{th1} and Lemma \ref{reop} imply a corollary.
\begin{coro}\label{co2}
Let $(X,d)$ be a locally compact, complete metric space. Let $F:X\rightarrow \K(X)$ be a multivalued map with an attractor $A^*$ and basin of attraction $\U$. 
There exists metric $D$ in $\K(X)$ equivalent with $d_h,$ such that  the operator $F:\K(X)\rightarrow \K(X)$ is a contraction in $\U$. 
\end{coro}
We immediately receive the following.

\begin{coro}\label{co3}
Let $(X,d)$ be a locally compact, complete metric space. Let $F:X\rightarrow \K(X)$ be a multivalued map with a strict  attractor $A^*$ and basin of attraction $B(A)$. 
There exists metric $D$ in $\K(X)$ equivalent with $d_H,$ such that  the operator $F:\K(X)\rightarrow \K(X)$ is a contraction in $\K(B(A))$. 
\end{coro}

The metric $D$ may be constructed by means of \cite{Op} or \cite{Ja}.
However, it need not be, in general, induced by any $d$ on $X$.
In general, stable attractors in compact metric spaces do not fulfill Fryszkowski condition.
\begin{exa}
Consider the following IFS.
$F=\{(X,d),\,f_1,\, f_2\},$
$X=\{(x,y)\in \R^2:x^2+y^2=1\},$ where $d$ is ordinary Euclidean metric and
$f_1(x,y):=(x,y),$
$f_2(x,y):=(x\cos {\alpha}-y\sin{\alpha},x\sin{\alpha}+y\cos{\alpha}),$
where $\alpha/\pi$ is irrational.
We will show that there is no metric $d'$ on $X$ equivalent to $d$ such that multivalued map $F:X\rightarrow \K(X)$ associated to the IFS is a contraction (w.r.t. $d'$ and $d_H'$). We will prove it by contradiction.

Suppose, there exists such metric $d'$. Then we can find a point $x_0\in X$ and $\epsilon>0$ such that $d'(x,x')\leq d(f_2(x),f_2(x')),$ $x,x'\in N_{\epsilon}(x_0)$. Otherwise, $f_2$ would possess periodic points according to \cite[4 Theorem 2]{Ed} and $d'$ would not be equivalent to $d$.

Since $0<\alpha<2\pi,$ we can choose $x,x'\in N_{\epsilon}(x_0)$ close enough so that 
$$d(f_2(x),\{x,x'\})\geq d(x,x') \ \mbox{and}\  d(f_2(x'),\{x,x'\})\geq d(x,x').$$ 
Then 
$$
d'_H(F(x),F(x'))=d'_H(\{x,f_2(x)\},\{x',f_2(x')\})=d'(f_2(x),f_2(x'))\geq d'(x,y),
$$
which is a contradiction to contractivity of $F$ w.r.t. $d'$ and $d'_H$.
\end{exa}

\begin{rema}
For $X=\R^2,$ the IFS does not possess an attractor.
\end{rema}

We can not drop the condition that the operator is generated by a multivalued map. 
In general, an attracting operator in compact metric space is not a contraction.
\begin{exa}
Let us denote by $\K_{C}(I)$ the set of closed subintervals of $I=[0,1],$ endowed with Euclidean metric.
The space  $(\K_{C}(I),d_H)$ is equivalent to a filled triangle $T=ABC,\,A=[0,1],\,B=[1,1],\,C=[0,0]$ in $\R^2$ endowed with the maximum metric (cf. e.g. \cite{AR2}). 

Inspired with an example in \cite[Theorem 10]{GKLOP} and \cite[Theorem 2.1]{Op}, we can construct an operator in $\K(I)$ which is attractive, but not contractive. 
We consider the space $T$ as a union of triangles $T_{\alpha},\,\alpha\in [0,1]$ with empty interiors and common point $A$ illustrated in Figure \ref{mcha}. Observe that there is a homotopy $h:T\times (0,1)\rightarrow T$ such that $T_{\alpha}=h(\alpha,T_1)$. 

Similarly, as discussed in \cite{BLR1},
Consider the set of points 
$X$ of the circle which may be
projected to real numbers and infinity on the real line $\R^*$. 
Let $f^*:\R^*\rightarrow \R^*$ be 
such that $f^*(x)=x+1$, $x\neq \infty$ and $f^*(\infty )=\infty $.

Observe that the induced map $f:X\to X$ is continuous with respect to 
the Euclidean metric on the circle. It is obvious that each point of $X$ 
is attracted to the point $\infty$ (see Figure \ref{co}). 

We can construct a homeomorphism $g:X\rightarrow T_{1}$ such that $A=\infty.$
Define  a map $G:T\rightarrow T$ such that $G(x)=h(\alpha)\circ g\circ h^{-1}(\alpha)$ for $x\in T_{\alpha}$ and $G(A)=A$.
$G$ is obviously continuous (see Figure \ref{co}) and  may be applied on $\K_C(I)$.     

Let us extend the map $G:\K_C(I)\rightarrow \K_C(I)$ to the whole hyperspace $\K(I)$.
Let us define a map $Q:\K(I)\rightarrow \K(I),$
 $$Q(D)=[\min(D),\max D],$$
which is continuous (even retraction). 

Defining $F=G\circ Q$ we obtain an operator in $\K(I)$ which has an attractor $I$. However, operator $F$  is not a contraction, since  
$F^n(\K(I))\nrightarrow \{I\}$ (see \cite{Op} and \cite[Theorem 10]{GKLOP}).

\begin{figure}[ht]
 \centerline{\resizebox{10cm}{!}{\includegraphics{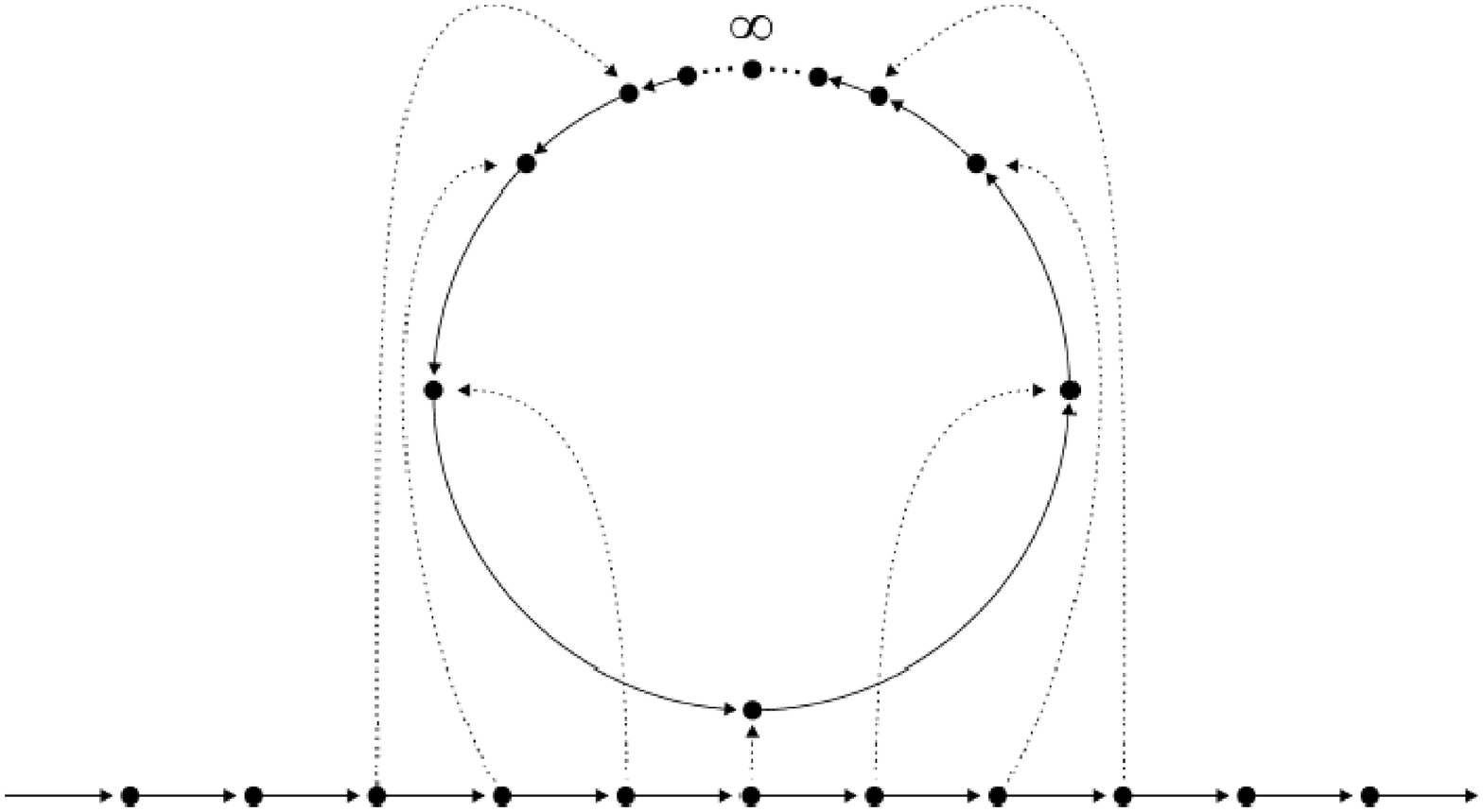}}}

\vspace{-12pt}
\caption{
}
\label{co}
\end{figure}

\begin{figure}[ht]
 \centerline{\resizebox{10cm}{!}{\includegraphics{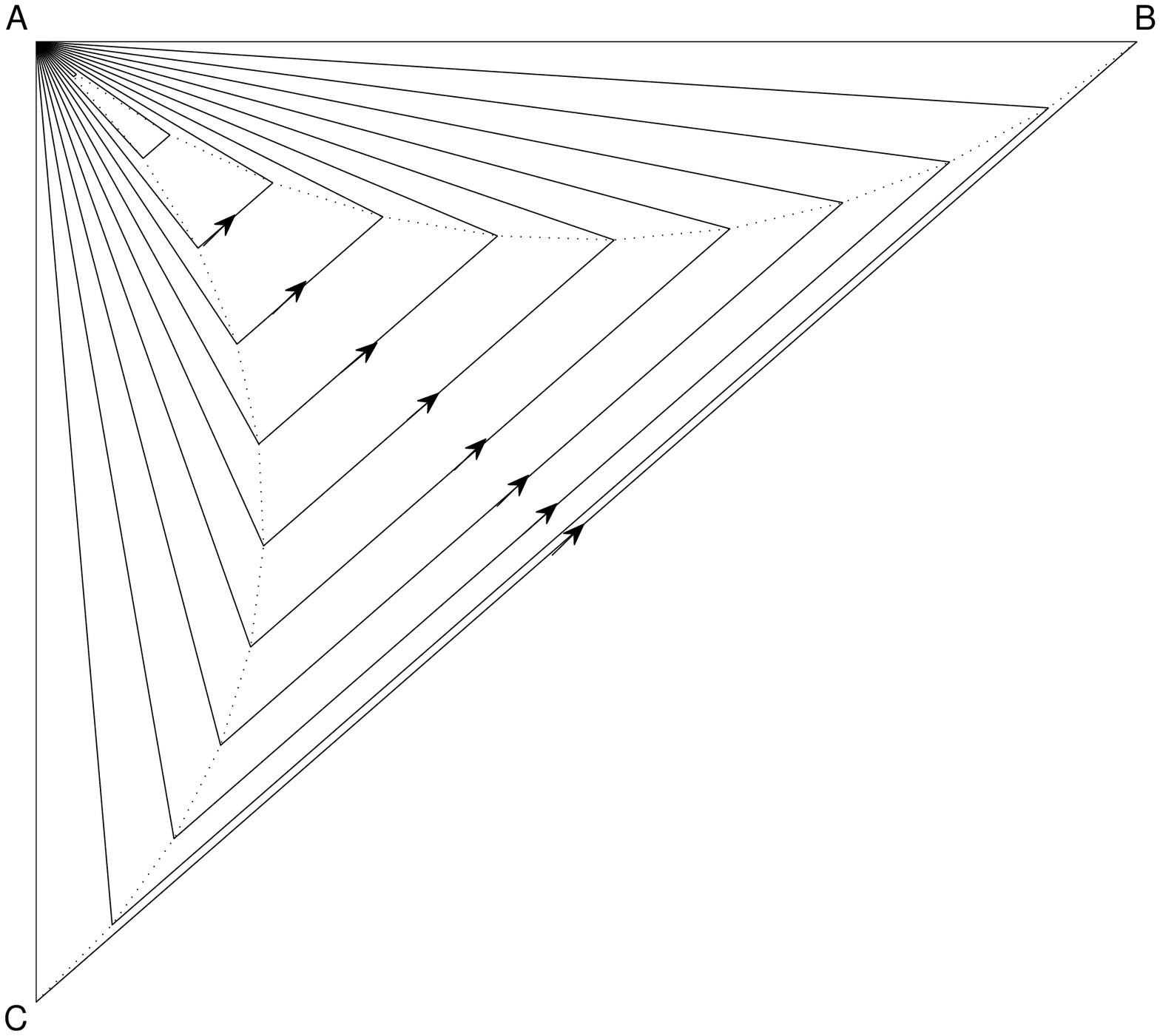}}}

\vspace{-12pt}
\caption{
}
\label{mcha}
\end{figure}

\end{exa}

\begin{rema}
Since attractors and strict attractors are topological notions and we employed topological means to prove stability of attractors, the whole proof of stability could be conducted even for topological spaces with a trade off for lower comprehensibility. 
\end{rema}

\end{document}